\documentclass[letterpaper, 10 pt, conference]{ieeeconf}  

\IEEEoverridecommandlockouts                              
\usepackage{amssymb}
\usepackage{tikz-cd}
\usetikzlibrary{arrows}

\usepackage{amsthm}
\usepackage{amsmath}
\usepackage{graphicx}
\usepackage{eucal}
\usepackage{subcaption} 
\usepackage{enumerate}

\newtheoremstyle{mystyle}
  {}
  {}
  {\itshape}
  {}
  {\bfseries}
  {.}
  { }
  {}

\theoremstyle{mystyle}

\newtheorem{Thm}{Theorem}[section]

\newtheorem{Def}[Thm]{Definition}
\newtheorem{Lem}[Thm]{Lemma}
\newtheorem{Prop}[Thm]{Proposition}

\newtheorem{Ex}[Thm]{Example}

\newtheorem{Rem}[Thm]{Remark} 
\newtheorem{Ass}[Thm]{Assumption}

\usepackage[backend=bibtex,
style=ieee,
bibencoding=ascii,
]{biblatex}
\makeatletter
\newbox\dottedarrow@box
\setbox\dottedarrow@box\hbox
  {%
    \begin{tikzpicture}
      \draw[dotted,->] (0,0) -- (1.5em,0);
    \end{tikzpicture}%
  }
\newcommand*\dottedarrow
  {\relax\ifmmode\expandafter\dottedarrow@m\else\expandafter\dottedarrow@t\fi}
\newcommand*\dottedarrow@t[1][1.5em]
  {\resizebox{#1}{!}{\raisebox{.5ex}{\usebox\dottedarrow@box}}}
\newcommand*\dottedarrow@m[1][]
  {%
    \if\relax\detokenize{#1}\relax
      \mathchoice
        {\dottedarrow@t}
        {\dottedarrow@t}
        {\dottedarrow@t[1.1em]}
        {\dottedarrow@t[0.9em]}%
    \else
      \dottedarrow@t[#1]%
    \fi
  }
\makeatother

\title{\LARGE \bf
On topological properties of compact attractors on Hausdorff spaces
}

\author{Wouter Jongeneel
\thanks{\today |The author is with the Risk Analytics and Optimization Chair, \'Ecole Polytechnique F\'ed\'erale de Lausanne, contact: \textit{wouter.jongeneel@epfl.ch}. This research was supported by the Swiss National Science Foundation under the NCCR \textit{Automation}, grant agreement 51NF40\_180545. The author is grateful to Matthew D. Kvalheim, Emmanuel Moulay and all reviewers for feedback and inspiration.} 
}

\begin{document}

\maketitle
\thispagestyle{empty}
\pagestyle{empty}

\begin{abstract}
We characterize when a compact, invariant, asymptotically stable attractor on a locally compact Hausdorff space is a strong deformation retract of its domain of attraction.
\end{abstract}

\section{Introduction}
\label{sec:intro}
The purpose of this note is to improve our understanding of topological properties of compact asymptotically stable attractors and their respective domain of attraction. Here, we will almost exclusively appeal to topological tools pioneered by Borsuk~\cite{ref:dydak2012ideas}. In particular, we will elaborate on the \textit{retraction theoretic} work by Moulay \& Bhat~\cite{ref:moulay2010topological}, which itself is a generalization of the seminal works~\cite[Thm.~21]{ref:sontag2013mathematical}, ~\cite{BhatBernstein} and~\cite[Thm.~4.1]{ref:bhatia1967asymptotic}.   

After Poincar\'e and Lyapunov (Liapunov), the modern qualitative study of attractors was largely propelled through the monographs by Birkhoff~\cite{ref:birkhoff1927dynamical} and Nemytskii \& Stepanov~\cite{ref:nemytskii}, with influential follow-up works by~Auslander, Bhatia \& Siebert~\cite{ref:auslanderbhatiasiebert1967asymptotic},  Wilson~\cite{ref:wilson1967structure}, Hahn~\cite{hahn1967stability}, Bhatia \& Szeg\"o~\cite{ref:bhatia1970stability}, Conley~\cite{ref:conley1978isolated}, Milnor~\cite{milnor1985concept} and many others, \textit{e.g.}, see~\cite[Ch.~1]{ref:jongeneel2023topological}.

Lately, attractors have been extensively studied through the lens of \textit{shape theory},~\textit{e.g.}, see~\cite{ref:garay1991strong,ref:giraldo2001some,ref:giraldo2009singular} and \cite[Prop.~1]{ref:kvalheim2022necessary}, with the seminal work of G\"unther \& Segal showing that a finite-dimensional compact subset of a manifold can be an attractor if and only if it has the shape of a finite polyhedron~\cite{ref:gunther1993every}. 

The interest in understanding topological properties of attractors and their respective domain of attraction stems from the simple observation that if a certain dynamical system does not exist, then certainly there is no feedback law resulting in a closed-loop dynamical system with precisely those dynamics.  

Indeed, this type of study often provides \textit{necessary} conditions of the form that an attractor must be \textit{equivalent} in some sense to its domain of attraction. With that in mind, one seeks a notion of equivalence that is weak enough to cover many dynamical systems, yet also strong enough to obtain insights,~\textit{e.g.}, obstructions. Hence, although shape equivalence is more widely applicable~\cite{ref:kapitanski2000shape} and in that sense more fundamental, we focus on \textit{homotopy equivalence} with the aim of recovering stronger necessary conditions.   

In the same spirit, by further restricting the problem class, one could even look for stronger notions of equivalence as recently done in~\cite{ref:yao2023domain}. There, in the context of a vector-field guided path-following problem, homotopy equivalences have been strengthened to topological equivalences (homeomorphisms).

Although we focus on continuous dynamical systems, one can link work of this form to families of differential inclusions~\cite{ref:mayhew2011topological}. Indeed, further partial generalizations of~\cite{ref:moulay2010topological} to nonsmooth dynamical system are presented in~\cite{ref:li2012morse}. 

\textit{Notation and technical preliminaries}:~The \textit{identity map} on a space $X$ is denoted by $\mathrm{id}_X$, that is, $\mathrm{id}_X:x\mapsto x$ $\forall x\in X$. The (embedded) $n$-sphere is the set $\mathbb{S}^n:=\{x\in \mathbb{R}^{n+1}:\langle x,x\rangle=1\}$, with the (closed) $n$-disk being $\mathbb{D}^n:=\{x\in \mathbb{R}^{n}:\langle x,x\rangle\leq 1\}$. The topological boundary of a space $X$ is denoted by $\partial X$,~\textit{e.g.},~$\partial \mathbb{D}^{n+1}=\mathbb{S}^n$. We use $\simeq_h$ to denote homotopy (equivalence), see Section.~\ref{sec:retraction}.  

A topological space $X$ is said to be a \textit{locally compact Hausdorff} space when: (i) for any $x\in X$ there is a compact set $K\subseteq X$ containing an open neighbourhood $U$ of $x$ and; (ii) for any $x_1,x_2\in X$ there are open neighbourhoods $U_1\ni x_1$ and $U_2\ni x_2$ such that $U_1\cap U_2 = \emptyset$,~\textit{e.g.}, see~\cite[p.~31, 104]{Lee1}. Examples of locally compact Hausdorff spaces are: $\mathbb{R}^n$, topological manifolds, the Hilbert cube, any discrete space and so forth. In particular, any compact Hausdorff space is locally compact. Regarding counterexamples, a space $X$ with the trivial topology $\tau=\{X,\emptyset\}$ is not Hausdorff and any infinite-dimensional Hilbert space is a Hausdorff topological vector space, yet, it fails to be locally compact, see also~\cite[Thm.~29.1]{munkres2014pearson}.

\section{Continuous dynamical systems}

In this note we study continuous (global) \textit{semi-dynamical systems} comprised of the triple $\Sigma:=(\mathsf{M},\varphi,\mathbb{R}_{\geq 0})$. Here,  $\mathsf{M}$ is a locally compact Hausdorff space and $\varphi:\mathbb{R}_{\geq 0}\times \mathsf{M}\to \mathsf{M}$ is a (global) semi-flow, that is, a continuous map that satisfies for any $x\in \mathsf{M}$:
\begin{enumerate}[(i)]
    \item $\varphi(0,x)=x$ (\textit{identity} axiom); and
    \item $(\varphi(s,\varphi(t,x))=\varphi(t+s,x)$ $\forall s,t \in \mathbb{R}_{\geq 0}:=\{t\in \mathbb{R}:t\geq 0\}$ (\textit{semi-group} axiom).
\end{enumerate}
We will usually write $\varphi^t$ instead of $\varphi(t,\cdot)$. 

We say that a point $x\in \mathsf{M}$ is a \textit{start point} (under $\Sigma$) if $\forall (t,y)\in \mathbb{R}_{>0}\times \mathsf{M}$ we have that $\varphi^t(y)\neq x$. Differently put, $x\in \mathsf{M}$ is a start point when a flow starting from $x$ cannot be extended backwards, see~\cite[Ex.~5.14]{ref:bhatiahajek2006local} for an example. To avoid confusion, the evaluation of an integral curve at $0$ is sometimes called a ``\textit{starting point}''~\cite[p.~206]{Lee2}, which is not what we are talking about here. Then, to eliminate the existence of start points we appeal to~\cite[Prop.~1.7]{ref:bhatiahajek2006local}, for instance, we can consider semi-flows generated by a smooth vector field. Concretely, let $F\in \Gamma^{\infty}(T\mathsf{M})$ be a smooth vector field on a smooth manifold $\mathsf{M}$. It is well-known that under these conditions, for each $p\in \mathsf{M}$ there is a $\varepsilon>0$ such that $\gamma:(-\varepsilon,\varepsilon)\to \mathsf{M}$ is an integral curve of $F$ with $\gamma(0)=p$~\cite[Prop.~9.2]{Lee2}, that is, in terms of the (local) flow $\varphi:(-\varepsilon,\varepsilon)\times \mathsf{M}\to \mathsf{M}$ we have
\begin{equation*}
    \left.\frac{\mathrm{d}}{\mathrm{d}t}\varphi^t(p)\right|_{t=s}=F\left(\varphi^s(p)\right),\quad s\in (-\varepsilon,\varepsilon).  
\end{equation*}
Hence, with the previous observation in mind we will assume the following throughout the remainder of the note. 
\begin{Ass}[Start points]
\label{ass:starting:points}
The set of start points under the semi-dynamical system $\Sigma$ is empty. 
\end{Ass} 

\subsection{Stability}
We will exclusively focus on a subclass of semi-dynamical systems with practically relevant stability properties. 

\begin{Def}[Attractor]
\label{def:attractor}
Given some global semi-dynamical system $\Sigma=(\mathsf{M},\varphi,\mathbb{R}_{\geq 0})$, then, a compact set $A\subseteq \mathsf{M}$ is said to be an invariant, local asymptotically stable, attractor when
\begin{enumerate}[(i)]
    \item $\varphi(\mathbb{R},A)=A$ (\textit{invariance}); and
    \item \label{local:stab}  for any open neighbourhood $U_{\varepsilon}\subseteq \mathsf{M}$ of $A$ there is an open neighbourhood $U_{\delta}\subseteq U_{\varepsilon}\subseteq \mathsf{M}$ of $A$ such that $\varphi(\mathbb{R}_{\geq 0},U_{\delta})\subseteq U_{\varepsilon}$~(\textit{Lyapunov stability}), plus, there is an open neighbourhood $W\subseteq \mathsf{M}$ of $A$ such that all semi-flows intialized in $W$ converge to $A$~(\textit{local attractivity}), that is, for any $p\in W$ and any open neighbourhood $V\subseteq \mathsf{M}$ of $A$ there is a $T\geq 0$ such that $\varphi^t(p)\in V$ $\forall t\geq T$.  
\end{enumerate}
\end{Def}
The combination of Lyapunov stability and local attractivity is referred to as \textit{local asymptotic stability}. When the neighbourhood $W$ in Item~\eqref{local:stab} can be chosen to be all of $\mathsf{M}$ we speak of \textit{global} asymptotic stability. 
Local asymptotic stability is also captured by the existence of an open neighbourhood $U\subseteq \mathsf{M}$ of $A$ such that $\cap_{t\geq 0}\varphi^t(U)=A$~\cite[Lem.~1.6]{ref:hurley1982attractors}

On can find several definitions of ``\textit{attractors}'' in the literature, see for instance~\cite[Def.~V.1.5]{ref:bhatia1970stability}, \cite{milnor1985concept} and~\cite[Sec.~2.2]{ref:kapitanski2000shape}. 

\begin{Def}[Domain of attraction]
\label{def:DOA}
    Let the compact set $A\subseteq \mathsf{M}$ be an invariant, local asymptotically stable attractor under the semi-dynamical system $\Sigma=(\mathsf{M},\varphi,\mathbb{R}_{\geq 0})$, then, its domain of attraction is
    \begin{align*}
        \mathcal{D}_{\Sigma}(A)=\{p\in \mathsf{M}\,:\,&\text{for any open neighbourhood }U\subseteq \mathsf{M}\\&\text{of }A \text{ there is a }T\geq 0 \text{ such that }\\&\varphi^t(p)\in U\, \forall t\geq T\}. 
    \end{align*}
\end{Def}
Definition~\ref{def:DOA} can be equivalently written in terms of convergent subsequences. Topological properties of attractors $A\subseteq \mathsf{M}$ and their respective domain of attraction $\mathcal{D}_{\Sigma}(A)\subseteq \mathsf{M}$ are an active topic of study since the 1960s~\cite{ref:bhatia1970stability,ref:jongeneel2023topological}.

To elaborate on the introduction, the interest stems from the observation that the (numerical) analysis or synthesis,~\textit{e.g.}, via feedback control, of dynamical systems $\Sigma$ can be involved, while topological properties of the pair $(\mathcal{D}_{\Sigma}(A),A)$ might be readily available. Here, topological knowledge of $\mathcal{D}_{\Sigma}(A)$ is frequently used to study if some ``\textit{desirable}'' domain of attraction is admissible. For instance, one can show that no point $p\in \mathbb{S}^1$ can be a global asymptotically stable attractor under any $\Sigma=(\mathbb{S}^1,\varphi,\mathbb{R}_{\geq 0})$,~\textit{e.g.}, see Theorem~\ref{thm:sontag} below. The intuition being that for this to be true the circle $\mathbb{S}^1$ needs to be torn apart, which is obstructed by demanding $\varphi$ to be continuous, see also~\cite[Fig.~1.1]{ref:jongeneel2023topological}. Again, we emphasize that conclusions of this form emerge without involved analysis of any particular system $\Sigma$.  

\subsection{Retraction theory}
\label{sec:retraction}
The previous example can be understood through $\mathbb{S}^1$ not being \textit{contractible}, that is, $\mathbb{S}^1$ is not \textit{homotopy equivalent} to a point $p$. Formally, two topological spaces $X$ and $Y$ are said to have the same \textit{homotopy type} when they are homotopy equivalent\footnote{More abstractly, homotopies are isomorphisms in the homotopy category of topological spaces.}, that is, there are continuous maps $f:X\to Y$ and $g:Y\to X$ such that $f\circ g\simeq_h \mathrm{id}_Y$ and $g\circ f\simeq_h \mathrm{id}_X$. In some sense, this notion is a more general version of a deformation retract---which we recall below, and is most naturally understood through invariance in differential-~\cite{ref:guillemin2010differential} and algebraic topology~\cite{Hatcher}. As alluded to, now, we recall that $A\subseteq X$ is a \textit{retract} of $X$ when there is a map $r:X\to A$ such that $r\circ \iota_A = \mathrm{id}_A$, for $\iota_A$ the inclusion map $\iota_A:A\hookrightarrow X$. The set $A$ is said be a \textit{deformation retract} of $X$ when $A$ is a retract and additionally $\iota_A\circ r \simeq_h \mathrm{id}_X$, implying that $X$ is homotopy equivalent to $A$. When, additionally, the homotopy is \textit{stationary relative to $A$}, we speak of a \textit{strong deformation retract}. 

\begin{Rem}[Deformation retracts]
The literature does not agree on what a ``\textit{deformation retract}'' is. For instance, Hatcher calls strong deformation retracts simply deformation retracts and speaks of ``deformation retracts in the weak sense'' where we would be speaking of simply a deformation retract~\cite[Ch.~0]{Hatcher}. This should be contrasted with for instance the 1965 text of Hu~\cite[Sec.~1.11]{ref:hu1965}. 
\end{Rem}

Next, a set $A\subseteq X$ is said to be a \textit{weak deformation retract} of $X$ when every open neighbourhood $U\supseteq A$ contains a strong deformation retract $V\supseteq A$ of $X$.   

In this note we will elaborate on the following result due to Moulay \& Bhat. In particular, we aim to understand when $\mathcal{D}_{\Sigma}(A)$ strongly deformation retracts onto $A$ and not just to a subset of a neighbourhood around $A$.
\begin{Thm}[{\cite[Thm.~5]{ref:moulay2010topological}}]
\label{thm:moulay-bhat}
Suppose that the compact set $A\subseteq \mathsf{M}$ is an invariant, local asymptotically stable attractor under the semi-dynamical system $\Sigma=(\mathsf{M},\varphi,\mathbb{R}_{\geq 0})$, then, $A$ is a weak deformation retract of $\mathcal{D}_{\Sigma}(A)$. 
\end{Thm}

It is well-known that when $\mathsf{M}$ is a smooth manifold and $A$ is an embedded submanifold of $\mathsf{M}$, then, $A$ is a strong neighbourhood deformation retract (see below for the definition) of $\mathsf{M}$ and thus $A$ is homotopic to $\mathcal{D}_{\Sigma}(A)$~\cite[Prop.~10]{ref:moulay2010topological}, see also~\cite{ref:lin2022wilson}. 
Our aim is to provide further, especially weaker, conditions for this to be true.

Indeed, Theorem~\ref{thm:moulay-bhat} can be seen as a generalization of the following well-known result due to Sontag.

\begin{Thm}[{\cite[Thm.~21]{ref:sontag2013mathematical}}]
\label{thm:sontag}
Suppose that the point $A:=\{p\}\subseteq \mathsf{M}$ is an invariant, global asymptotically stable attractor under the semi-dynamical system $\Sigma=(\mathsf{M},\varphi,\mathbb{R}_{\geq 0})$, then, $\mathcal{D}_{\Sigma}(A)$ is contractible. 
\end{Thm}

We aim to strengthen this generalization. We do this by appealing to neighbourhood retracts. A set $A\subseteq X$ is a \textit{neighbourhood retract} of $X$ when there is an open neighbourhood $U\subseteq X$ of $A$ such that $A$ is retract of $U$. This definition extends naturally to (strong) deformation retracts. Now indeed, if for instance $\mathsf{M}$ weakly deformation retracts onto $A\subseteq \mathsf{M}$ while $A$ is a neighbourhood deformation retract of $\mathsf{M}$, then $\mathsf{M}$ deformation retracts onto $A$ and thus $A\simeq_h \mathsf{M}$~\cite[Thm.~4]{ref:moulay2010topological}.

In this note we focus on homotopy equivalence, a similar but weaker notion is that of \textit{shape equivalence}, understood as being for \v{C}ech (co)homology what homotopy theory is for singular (co)homology.  Indeed, only for sufficiently ``\textit{nice}'' spaces, singular cohomology and \v{C}ech cohomology agree. For a concise introduction to shape theory, in the sense of Fox~\cite{ref:rh1972shape}, we refer the reader to~\cite[Sec.~3]{ref:kapitanski2000shape}. The crux is to work with open neighbourhoods of a set and not solely with the set itself. For instance, the Warsaw circle $\mathbb{W}^1$, as studied below in Example~\ref{ex:Sw} is not homotopy equivalent to the circle $\mathbb{S}^1$ but the two spaces are shape equivalent.  


\section{Cofibrations}
It follows from Theorem~\ref{thm:moulay-bhat} that for $A$ to be homotopy equivalent to $\mathcal{D}_{\Sigma}(A)$ , it suffices for $A$ to be a neighbourhood deformation retract of $\mathcal{D}_{\Sigma}(A)$. We will appeal to \textit{cofibrations} to capture this property. 

The theory of cofibrations emerges from the so-called ``\textit{extension problem}'', that is, understanding when a continuous map $f:A\subseteq X\to Y$ can be extended from $A$ to all of $X$. A typical counterexample is any map $f:\partial \mathbb{D}^{n+1}\to Y$ of nonzero degree (for $Y$ compact and with $\mathrm{dim}(Y)=n)$, that is, when $\mathrm{deg}_2(f)\neq 0$, $f$ cannot be extended from the $n$-sphere $\mathbb{S}^{n}\simeq \partial \mathbb{D}^{n+1}$ to the $(n+1)$-disk $\mathbb{D}^{n+1}$~\cite[Sec.~2.4]{ref:guillemin2010differential}.

Now, cofibrations tell us, loosely speaking, if maps that can be extended, lead to homotopies that can be extended. To define this, we need the following. 
Let $X$ be a topological space and $A\subseteq X$, then, a pair $(X,A)$ has the \textit{homotopy extension property} (HEP) when, for any $Y$, the diagram 
\begin{equation}
\label{equ:diag1}
\begin{tikzcd}[row sep=small, column sep=large]
(A\times [0,1]) \cup (X \times \{0\}) \arrow{r}{} \arrow[hookrightarrow]{d} & Y  \\
X\times [0,1] \arrow[dotted, bend right]{ur}
\end{tikzcd}
\end{equation}
can always be completed (``completed'' means that the \textit{dotted arrow} \dottedarrow\ can be found) to be commutative. Differently put, given a a homotopy $H:A\times [0,1]\to Y$ and some map $g:X\to Y$ such that $H(\cdot,0)=g|_A$, one needs to be able to extend the homotopy from $A$ to $X$.
Pick $Y=(A\times [0,1]) \cup (X \times \{0\})$, then we see that $(X,A)$ having the HEP implies that $(A\times [0,1]) \cup (X \times \{0\})$ is a retract of $X\times [0,1]$. On the other hand, one can show that the existence of such a retract implies that $(X,A)$ has the HEP, that is, these two notions are equivalent, see Theorem~\ref{thm:cofib:retract:ndr} below. See also~\cite[p.~13]{strom1966note} for a stronger result. 

Then, a continuous map $i:A\to X$ is said be a \textit{cofibration}\footnote{Our notion of cofibration is aligned with the so-called \textit{Hurewicz cofibration}. We will not discuss \textit{Serre cofibrations}, which are most easily understood through \textit{fibrations}, a notion dual to that of cofibrations,~\textit{e.g.}, see~\cite[Ch.~7]{may1999concise}.} if the following commutative diagram 
\begin{equation}
\label{equ:diag2}
\begin{tikzcd}[row sep=small,column sep=large]
A \times \{0\} \arrow{d} \arrow[hookrightarrow]{r}               & A\times [0,1] \arrow{d} \arrow[rdd, bend left] &   \\
X\times \{0\} \arrow[hookrightarrow]{r} \arrow[rrd, bend right] & X\times [0,1] \arrow[rd, dotted]               &   \\
                                    &                                    & Y
\end{tikzcd}
\end{equation}
can be completed for any $Y$. Loosely speaking, the map $i$ is a cofibration when it has the HEP\footnote{We focus on pairs $(X,A)$ such that $A\subseteq X$, this inclusion is, however, not required for a cofibration to be well-defined. Nevertheless, to make sense of~\eqref{equ:diag1} one should work with $(A\times [0,1]) \cup_{i} X$ instead of $(A\times [0,1]) \cup (X \times \{0\})$, that is, with the so-called ``\textit{mapping cylinder}'' as further discussed below.}. 

Next, we need a slight variation of the aforementioned notions of retraction, that of a \textit{neighbourhood deformation retract pair} (NDR pair).

\begin{Def}[NDR pair]
\label{def:NDR}
A pair $(X,A)$ is said to be an NDR pair if:
\begin{enumerate}[(i)]
    \item there is a continuous map $u:X\to [0,1]$ such that $A=u^{-1}(0)$; and
    \item there is a homotopy $H:X\times [0,1]\to X$ such that $H(\cdot,0)=\mathrm{id}_X$, $H(x,s)=x$ $\forall x\in A$, $\forall s\in [0,1]$ and $H(x,1)\in A$ if $u(x)<1$.
\end{enumerate} 
\end{Def}

See that if $u(x)<1$ $\forall x\in X$, then, $A$ is a strong deformation retract of $X$. In general, however, we cannot assume $u$ to be of this form. See that for $(X,A)$ to be an NDR pair, $A$ must be closed. Now, a useful result is the following.
\begin{Thm}[{\cite[Ch.~6]{may1999concise}}]
\label{thm:cofib:retract:ndr}
    Let $A$ be closed in $X$, then, the following are equivalent:
    \begin{enumerate}[(i)]
        \item the inclusion $\iota_A:A\hookrightarrow X$ is a cofibration;
        \item \label{item:retract} $(A\times [0,1])\cup (X\times\{0\})$ is a retract of $X\times [0,1]$;
        \item \label{item:NDR} $(X,A)$ is an NDR pair. 
    \end{enumerate}
\end{Thm}
To be somewhat self-contained, we provide intuition regarding Item~\eqref{item:retract} and Item~\eqref{item:NDR}. In particular, we highlight continuity. 
\begin{proof}[Proof (sketch)]
    Suppose we have the retract $r:X\times [0,1]\to (A\times [0,1])\cup (X\times \{0\})$. Define the projections $\pi_1:X\times [0,1]\to X$ and $\pi_2:X\times [0,1]\to [0,1]$. Next, define the map $u:X\to [0,1]$ via
    \begin{equation}
    \label{equ:u}
        u(x) = \sup_{\tau\in [0,1]}\{\tau-\pi_2 ( r(x,\tau))\},
    \end{equation}
    where the supremum is attained since $[0,1]$ is compact and $\pi_2$ and $r$ are continuous. Now define the homotopy $H:X\times [0,1]\to X$ by $H(x,s)=\pi_1( r(x,s))$. Indeed, one can readily check that the pair $(u,H)$ satisfies the properties required for $(X,A)$ to be an NDR pair. Note that since the retract $r$ is continuous, we have that $u$ is not identically $0$ when $X\setminus A \neq \emptyset$. The map $u$ constructed through~\eqref{equ:u} is in fact continuous since $[0,1]$ is compact and both $\pi_2$ and $r$ are continuous,~\textit{e.g.}, one can appeal to the simplest setting of Berge's maximum theorem~\cite{ref:berge1963topological}. 
\end{proof}

Note that in general, $(A\times [0,1])\cup (X\times\{0\})$ will be equivalent to the \textit{mapping cylinder} under the inclusion map $\iota_A:A\hookrightarrow X$, that is, $M{\iota_A}=( (A\times [0,1]) \cup X)/\sim$ with $(a,0)\sim \iota_A(a)$ for all $a\in A$, also denoted by $(A\times [0,1])\cup_{\iota_A}X$. Equivalence can possibly fail when the product and quotient topologies under consideration do not match.  

For illustrative purposes, we end this section with the collection of a powerful result. Omitting the details, we recall that $X$ is a \textit{CW complex} when $X$ can be constructed via iteratively ``\textit{glueing}'' $n$-cells, being topological disks $\mathbb{D}^n$, along their boundary to a $(n-1)$-dimensional CW complex, with a $0$-dimensional CW complex being simply a set of discrete points. For instance, the circle $\mathbb{S}^1$ can be constructed from a single point and the interval. Then, a set $A\subseteq X$ is a \textit{subcomplex} of the CW complex $X$ when it is closed and a union of cells of $X$. For more on CW complices, we refer to~\cite[Ch.~0]{Hatcher} and~\cite[Ch.~5]{Lee1}.

\begin{Prop}[CW complices~{\cite[Prop.~0.16]{Hatcher}}]
\label{prop:CW:cofib}
Let $X$ be a CW complex and $A\subseteq X$ a subcomplex, then, the inclusion map $\iota_A:A\hookrightarrow X$ is a cofibration. 
\end{Prop}

Proposition~\ref{prop:CW:cofib} hinges on $\iota_{\mathbb{S}^{n-1}}:\mathbb{S}^{n-1}\hookrightarrow \mathbb{D}^{n}$ being a cofibration, which can be shown via showing that $(\mathbb{D}^n,\mathbb{S}^{n-1})$ is an NDR pair, however, using a strategy of more general use, one can show there is a (strong deformation) retract from $\mathbb{D}^n\times [0,1]$ onto $(\partial \mathbb{D}^n\times [0,1])\cup (\mathbb{D}^n\times \{0\})$~\cite[p.~15]{Hatcher},~\textit{e.g.}, consider some point $(0,c)\in \mathbb{D}^n\times \mathbb{R}_{\geq 2}$ and project $(0,c)$ onto $(\partial \mathbb{D}^n\times [0,1])$, then, the line between $(0,c)$ and this projection provides for the homotopy. 

CW complices are fairly general, yet, properties that obstruct $X$ admitting a CW decomposition are for instance: (i) $X$ failing to be locally contractible~\cite[Prop.~A4]{Hatcher}; and (2) $X$ failing to adhere to Whitehead's theorem~\textit{cf.}~Example~\ref{ex:Sw}.  

\subsection{Main result}
Now, we have collected all ingredients to prove the following. 
\begin{Lem}[$\Leftarrow$~Cofibration]
\label{lem:cofib1}
    Let $A\subseteq \mathsf{M}$ be a compact, invariant, asymptotically stable, attractor with domain of attraction $\mathcal{D}_{\Sigma}(A)$. If $\iota_A:A\hookrightarrow \mathcal{D}_{\Sigma}(A)$ is a cofibration, then, $A$ is a strong deformation retract of $\mathcal{D}_{\Sigma}(A)$.
\end{Lem}
\begin{proof}
    We know from Theorem~\ref{thm:moulay-bhat} that $A$ is a weak deformation retract of $\mathcal{D}_{\Sigma}(A)$. Since $\iota_A:A\hookrightarrow \mathcal{D}_{\Sigma}(A)$ is a cofibration, we also know from Theorem~\ref{thm:cofib:retract:ndr} that $(\mathcal{D}_{\Sigma}(A),A)$ is an NDR pair. Then, recall Definition~\ref{def:NDR} and recall the proof of Theorem~\ref{thm:cofib:retract:ndr}. Now, let $W:=u^{-1}([0,1))\supset A$, which is open since $u:\mathcal{D}_{\Sigma}(A)\to [0,1]$ is continuous, and consider the map $H|_{W\times [0,1]}$. It is imperative to remark that this map does \textit{not} provide a strong deformation retract from $W$ onto $A$ in general. The reason why we cannot conclude on the existence of such a map is that we cannot guarantee that throughout the homotopy we have $H(x,s)\in W$ for any $(x,s)\in W\times  [0,1]$. Indeed, we have a map $H|_{W\times [0,1]}:W\times [0,1]\to \mathcal{D}_{\Sigma}(A)\supseteq W$, the codomain cannot be assumed to be $W$. Precisely this detail was already known to Str{\o}m~\textit{cf.}~\cite[Thm.~2]{strom1966note}, see also~\cite[p.~432]{ref:bredon1993topology}.
    Nevertheless, since $A$ is a weak deformation retract of $\mathcal{D}_{\Sigma}(A)$ we know that $W$ contains a set $V\supseteq A$ such that $\mathcal{D}_{\Sigma}(A)$ strongly deformation retracts onto $V$, that is, there is map $\Bar{H}:\mathcal{D}_{\Sigma}(A) \times [0,1]\to \mathcal{D}_{\Sigma}(A)$ such that $\Bar{H}(x,0)=x$ $\forall x\in \mathcal{D}_{\Sigma}(A)$, $\Bar{H}(x,1)\in V$ $\forall x \in \mathcal{D}_{\Sigma}(A)$ and $\bar{H}(x,s)=x$ $\forall (x,s)\in V\times [0,1]$. Hence, the continuous map $\widetilde{H}:\mathcal{D}_{\Sigma}(A)\times [0,1]\to \mathcal{D}_{\Sigma}(A)$ defined by 
\begin{equation*}
    \widetilde{H}(x,s) = \begin{cases}
\bar{H}(x,2s) \quad & s\in [0,\tfrac{1}{2}]\\
{H}\left(\bar{H}(x,1),2s-1\right) \quad & s\in (\tfrac{1}{2},1]
    \end{cases}
\end{equation*}
is a homotopy and provides for the strong deformation retract of $\mathcal{D}_{\Sigma}(A)$ onto $A$. 
\end{proof}

To continue, we need a converse result, we emphasize $\Sigma$. 
\begin{Lem}[$\Rightarrow$~Cofibration]
\label{lem:cofib2}
    Suppose that $\mathsf{M}$ is a locally compact Hausdorff space, that $\Sigma$ satisfies Assumption~\ref{ass:starting:points} and let $A\subseteq \mathsf{M}$ be a compact, invariant, asymptotically stable, attractor with domain of attraction $\mathcal{D}_{\Sigma}(A)$. If $A$ is a strong deformation retract of $\mathcal{D}_{\Sigma}(A)$, then, $\iota_A:A\hookrightarrow \mathcal{D}_{\Sigma}(A)$ is a cofibration.
\end{Lem}
\begin{proof}
We will appeal to the characterization of a cofibration as given by Theorem~\ref{thm:cofib:retract:ndr}.
    As $A$ is a strong deformation retract of $\mathcal{D}_{\Sigma}(A)$ by assumption, then, to conclude on $(\mathcal{D}_{\Sigma}(A),A)$ being an NDR pair, we need to construct the map $u:\mathcal{D}_{\Sigma}(A)\to [0,1]$. As $A$ is a compact, invariant, asymptotically stable attractor, $\mathsf{M}$ is a locally compact Hausdorff space and $\Sigma$ satisfies Assumption~\ref{ass:starting:points}, there is a Lyapunov function of precisely this form~\cite[Thm.~10.6]{ref:bhatiahajek2006local}.  
\end{proof}

\begin{Thm}[Cofibrations]
\label{thm:cofib}
    Suppose that $\mathsf{M}$ is a locally compact Hausdorff space, that $\Sigma$ satisfies Assumption~\ref{ass:starting:points} and let $A\subseteq \mathsf{M}$ be a compact, invariant, asymptotically stable, attractor with domain of attraction $\mathcal{D}_{\Sigma}(A)$. Then, $A$ is a strong deformation retract of $\mathcal{D}_{\Sigma}(A)$ if and only if the inclusion $\iota_A:A\hookrightarrow \mathcal{D}_{\Sigma}(A)$ is a cofibration,
\end{Thm}
\begin{proof}
    The results follow directly by combining Lemma~\ref{lem:cofib1} and Lemma~\ref{lem:cofib2}. 
\end{proof}

\subsection{Examples}

Regarding ramifications of Theorem~\ref{thm:cofib}, we start with a sanity check. We know that for a a linear ODE $\dot{x}=Fx$ with $F\in \mathbb{R}^{n\times n}$ a Hurwitz matrix, $A=\{0\}$ and $\mathcal{D}_{\Sigma}(A)=\mathbb{R}^n$. Hence, we remark that: (i) $\iota_{\{0\}}:\{0\}\hookrightarrow \mathbb{R}^n$ is a cofibration,~\textit{e.g.}, since $(\mathbb{R}^n,\{0\})$ is an NDR pair under the map $x\mapsto u(x):=1-e^{-\langle x,x\rangle}$; and (ii) $\mathbb{R}^n$ strongly deformation retracts onto $0\in \mathbb{R}^n$ via the map $\mathbb{R}^n\times [0,1]\ni (x,s) \mapsto H(x,s):= (1-s)\cdot x$. 

Cofibrations that are not strong deformation retracts are abundant. We start with a well-known example.  
\begin{Ex}[Spheres and disks]
    It can be shown that the inclusion $\iota_{\mathbb{S}^n}:\mathbb{S}^n \hookrightarrow \mathbb{D}^{n+1}$ is a cofibration,~\textit{e.g.}, see the remark on CW complices above. However, $\mathbb{D}^{n+1}$ cannot strongly deformation retract onto $\mathbb{S}^n$ since $\mathbb{S}^n$ and $\mathbb{D}^{n+1}$ are not even homotopy equivalent,~\textit{e.g.}, $\chi(\mathbb{S}^n)\neq \chi(\mathbb{D}^{n+1})$. Hence, $\mathbb{S}^n$ cannot be a global, asymptotically stable, attractor under any semi-dynamical system $\Sigma=(\mathbb{D}^{n+1},\varphi,\mathbb{R}_{\geq 0})$. Concurrently, this example illustrates that the two conditions from Theorem~\ref{thm:cofib} are truly distinct. 
\end{Ex}

We proceed with an example where we obtain topological insights through dynamical systems knowledge.

\begin{Ex}[The Warsaw circle]
\label{ex:Sw}
Let $\mathbb{W}^1:=\{(0,x_2)\in \mathbb{R}^2:x_2\in [-1,1]\}\cup \{(x_1,\sin(x_1^{-1}))\in \mathbb{R}^2:x_1\in (0,\pi^{-1}) \}\cup \{\text{arc from }(0,-1) \text{ to }(\pi^{-1},0)\}$ denote the so-called ``\textit{Warsaw circle}''. The set $\mathbb{W}^1$ is compact, but not a manifold since $\mathbb{W}^1$ is not locally connected.
Hastings showed\footnote{Although a substantial part of the proof is left to the reader.} that $\mathbb{W}^1$ can be rendered a compact, invariant, locally asymptotically stable attractor with an annular neighbourhood $A\subset \mathbb{R}^2$ as $\mathcal{D}_{\Sigma}(\mathbb{W}^1)$~\cite{ref:hastings1979higher}. Although the circle $\mathbb{S}^1\subset \mathbb{R}^2$ and $\mathbb{W}^1\subset \mathbb{R}^2$ are shape equivalent, they are not homotopy equivalent since $\pi_1(\mathbb{W}^1)\simeq 0$ while $\pi_1(\mathbb{S}^1)\simeq \mathbb{Z}$ and the fundamental group $\pi_1(\cdot)$ is homotopy invariant~\cite[Thm.~7.40]{Lee1}. As such, $\mathbb{W}^1$ cannot be a strong deformation retract of any annulus $A\subset \mathbb{R}^2$ it embeds in. 
Then, according to Theorem~\ref{thm:cofib}, $\iota_{\mathbb{W}^1}:\mathbb{W}^1\hookrightarrow A$ cannot be a cofibration.  
\end{Ex}

We recall that for inclusion maps $\iota_A:A\hookrightarrow X$ to not be a cofibration, the pair $(X,A)$ cannot be too regular,~\textit{e.g.}, by Proposition~\ref{prop:CW:cofib} $(X,A)$ cannot be a CW pair. Indeed, by the Whitehead theorem~\cite[Thm.~4.5]{Hatcher}, The Warsaw circle $\mathbb{W}^1$ is not homotopy equivalent to a CW complex. This could also be concluded by observing that CW complices must be locally path-connected. 

Next, we provide an example inspired by an example from~\cite[p.~78--79]{ref:homotopietheorie}. Here we gain dynamical insights via topological knowledge. Before doing so, we recall the difference between the \textit{box} and \textit{product} topology. Let $X_{\alpha}$ be a topological space indexed by $a\in A$, then, if we endow a topological space of the form $X=\prod_{\alpha\in A}X_{\alpha}$ with the product topology, open sets are of the form $\prod_{\alpha\in A}U_{\alpha}$ with $U_{\alpha}$ open in $X_{\alpha}$ and all but finitely many $U_{\alpha}=X_{\alpha}$. The box topology, on the other hand, does not require the last constraint to hold and open sets are simply of the form $\prod_{\alpha\in A}U_{\alpha}$ with $U_{\alpha}$ open in $X_{\alpha}$. When $A$ is finite, these topologies are equivalent, however, the box topology is finer than the product topology in general.

\begin{Ex}[The Tychonoff cube]
\label{ex:tychonoff}
    Let $\Omega > \aleph_0$, then, define the Tychonoff cube as $[0,1]^{\Omega}$, that is, as a uncountably infinite product of the unit interval. Here we endow $[0,1]$ with the standard topology and $[0,1]^{\Omega}$ with the product topology. As such, $[0,1]^{\Omega}$ is a compact Hausdorff space by Tychonoff's theorem~\cite[Thm.~37.3]{munkres2014pearson} and the fact that any product of Hausdorff spaces is Hausdorff~\cite[Thm.~19.4]{munkres2014pearson}. Exploiting compactness, $\{0\}^{\Omega}\in [0,1]^{\Omega}$ can be shown to be a strong deformation retract of $[0,1]^{\Omega}$. Indeed, one can simply use the map $[0,1]^{\Omega}\times [0,1] \ni (x,s)\mapsto H(x,s):=(1-s)\cdot x$, which would be continuous in the product topology, but not in the box topology. Despite the strong deformation retraction, $\iota_{0}:\{0\}^{\Omega}\hookrightarrow [0,1]^{\Omega}$ is not a cofibration since otherwise, by Definition~\ref{def:NDR} and Theorem~\ref{thm:cofib:retract:ndr}, there must be a continuous map $u:[0,1]^{\Omega}\to [0,1]$ such that $\{0\}^{\Omega}=u^{-1}(0)$. However, it can be shown that such a map fails to exist due to $\Omega$ being uncountable~\cite[p.~78--79]{ref:homotopietheorie}, this is where the product topology enters. Hence, Theorem~\ref{thm:cofib} implies that $\{0\}^{\Omega}\in [0,1]^{\Omega}$ cannot be an asymptotically stable attractor, for any continuous---with respect to the product topology on $[0,1]^{\Omega}$ ---semi-dynamical system $\Sigma=([0,1]^{\Omega},\varphi,\mathbb{R}_{\geq 0})$. Note that if $\Omega$ would be finite, then, the map $u$ does exist and can be chosen to be $u:(x_1,\dots,x_{\Omega})\mapsto \max_{i=1,\dots,\Omega} \{x_i\}$. 
\end{Ex}

Note that Example~\ref{ex:tychonoff} is essentially saying that despite seemingly convenient properties of $\Sigma=([0,1]^{\Omega},\varphi,\mathbb{R}_{\geq 0})$, a Lyapunov function fails to exist for $\{0\}^{\Omega}\in [0,1]^{\Omega}$. Concurrently, this example shows that even a strong notion of homotopy equivalence can be insufficient to conclude on the existence of an asymptotically stable attractor. Examples of this form obstruct continuous \textit{stabilization} as well. 

Although, in general, a metrizable space must be merely countably locally finite ($\sigma$-locally finite)~\cite[Thm.~40.3]{munkres2014pearson}, compact metric spaces must be second countable. Hence, $[0,1]^{\Omega>\aleph_0}$ is not metrizable since $\Omega>\aleph_0$ obstructs second countability. Similarly, one can consider the topology of pointwise convergence. Regardless, Example~\ref{ex:tychonoff} illustrates where to look for counterexamples. Indeed, as $[0,1]^{\Omega>\aleph_0}$ is not a normed space and in particular not a Hilbert space, it does not fit into common analysis frameworks, \textit{e.g.}, \cite{ref:curtain2012introduction}. 

It turns out that Theorem~\ref{thm:cofib} covers known results in case $A$ is an embedded submanifold of $\mathsf{M}$. We will assume all our manifolds under consideration to be $C^{\infty}$-smooth and second countable. In that case, let $A\subseteq \mathsf{M}$ be a closed embedded submanifold, then, one appeals to the existence of a \textit{tubular neighbourhood}~\cite[Thm.~6.24]{Lee2} to show that $(\mathcal{D}_{\Sigma}(A),A)$ comprises an NDR pair. Hence, using the following proposition, \cite[Prop.~10]{ref:moulay2010topological} follows as a corollary to Theorem~\ref{thm:cofib}, see also \cite{ref:lin2022wilson}. 
\begin{Prop}[Submanifolds]
\label{prop:submf:cofib}
    Let $A$ be a compact embedded submanifold of $\mathsf{M}$, for $\mathsf{M}$ as in the paragraph above, then, $\iota_A:A\hookrightarrow \mathsf{M}$ is a cofibration. 
\end{Prop}

Our last example pertains to compositions, indicating that Theorem~\ref{thm:cofib} can be applied to subsystems. 
\begin{Ex}[Compositions]
    Cofibrations are closed under composition. Let $i_1:A\to B$ and $i_2:B\to C$ be cofibrations, then, $i:=i_2\circ i_1:A\to C$ is a cofibration. To see this, consider the diagram
    \begin{equation}
    \label{equ:comp}
\begin{tikzcd}[row sep=small,column sep=large]
A\times \{0\} \arrow{d} \arrow[hookrightarrow]{r}                         & A\times [0,1] \arrow{d} \arrow[bend left=49]{rddd}{\alpha_1} &   \\
B\times \{0\} \arrow{d} \arrow[hookrightarrow]{r} \arrow[bend right]{rrdd}{\alpha_2} & B\times [0,1] \arrow{d} \arrow[bend left]{rdd}{\beta_1}  &   \\
C\times \{0\} \arrow[hookrightarrow]{r} \arrow[bend right]{rrd}{\beta_2}            & C\times [0,1] \arrow[rd, dotted]                &   \\
                                               &                                     & Y
\end{tikzcd}.
    \end{equation}
    For any triple $(Y,\alpha_1,\alpha_2)$ there is some appropriate map $\beta_1$ completing~\eqref{equ:diag2} for the cofibration $i_1:A\to B$, but for any triple $(Y,\beta_1,\beta_2)$ the diagram~\eqref{equ:comp} can be completed since $i_2:B\to C$ is a cofibration. 
\end{Ex}

\section{Future work}
Several directions of future work are: (i) discrete systems of the form $\Sigma=(\mathsf{M},\varphi,\mathbb{Z}_{\geq 0})$; (ii) exploiting duality (fibrations); (iii) extensions to other notions of stability; (iv) developing computational tools (homology); (v) relaxing the invariance assumption; and (vi) addressing stochastic systems in a meaningful way. 

\subsection{Closed attractors}
Several results hold when the compactness assumption on $A\subseteq \mathsf{M}$ is relaxed to $A$ being merely \textit{closed},~\textit{e.g.}, see~\cite{ref:lin1996smooth,ref:bhatiahajek2006local}, this generalization is future work. We do emphasize that the generalization is not trivial. Consider for instance~\cite[Ex.~22]{ref:lin2022wilson} with $\mathsf{M}=\mathbb{R}^2\setminus \{(1,0)\}$ and $A=\mathbb{S}^1\setminus\{(1,0)\}\subset \mathsf{M}$. There, the authors construct a vector field on $\mathsf{M}$ such that $A$ is an asymptotically stable attractor, with $\mathcal{D}_{\Sigma}(A)=\mathsf{M}\setminus \{(0,0)\}$. So, although $A$ is an attractor and $\iota_A:A\hookrightarrow \mathcal{D}_{\Sigma}(A)$ is a cofibration due to Proposition~\ref{prop:submf:cofib}, $A$ cannot be a strong deformation retract of $\mathcal{D}_{\Sigma}(A)$ since those sets are not homotopy equivalent. Indeed, $A$ is \textit{not} compact and one cannot simply appeal to Theorem~\ref{thm:moulay-bhat}. The intuition is that for attractors of this form, limits need not be attained and as such stability does not provide for a homotopy between $A$ and $\mathcal{D}_{\Sigma}(A)$. Formally speaking, the proof of Theorem~\ref{thm:moulay-bhat} exploits compactness of the sublevel sets of the corresponding Lyapunov function and implicitly \textit{Cantor's intersection theorem}, which fails to be generally true for closed sets. See also~\cite[Counterex.~1]{ref:yao2023domain}.  

\printbibliography

\end{document}